\documentclass[a4paper, 12pt, reqno]{amsart}
\usepackage{amsmath}
\usepackage{amsfonts}
\usepackage{amssymb}
\usepackage{xcolor}

\newcommand{\I}{\mathcal{I}}
\newcommand{\D}{\mathcal{D}}
\newcommand{\N}{\mathbb{N}}
\newcommand{\Q}{\mathcal{Q}}
\newcommand{\R}{\mathbb{R}}
\newcommand{\Z}{\mathbb{Z}}

\DeclareMathOperator\diam{diam}
\newtheorem{lemma}{Lemma}
\newtheorem{theorem}{Theorem}

\newtheorem{proposition}{Proposition}
\begin{document}
\title[Accessible values of Assouad dimension]{Accessible values of Assouad and the lower dimensions of subsets}
\author[C.  Chen \and M. Wu \and W. Wu]{Changhao Chen \and Meng Wu \and Wen Wu}
\address[C. Chen]{Department of Mathematical Sciences, University of Oulu, P.O. Box 3000, 90014 Oulu, Finland. }\email{changhao.chen@oulu.fi}
\address[M. Wu]{Department of Mathematical Sciences, University of Oulu, P.O. Box 3000, 90014 Oulu, Finland.}\email{meng.wu@oulu.fi}
\address[W. Wu]{ Faculty of Mathematics and Statistics, Hubei University, 430062 Wuhan, P.R. China\hfill\newline
\indent Department of Mathematical Sciences, University of Oulu, P.O. Box 3000, 90014 Oulu, Finland. }\email{hust.wuwen@gmail.com}


\thanks{We acknowledge the support of Academy of Finland, the Centre of Excellence in Analysis and Dynamics Research. Changhao Chen acknowledges the support of the Vilho, Yrj\"o, and Kalle V\"ais\"al\"a foundation. Wen Wu was also supported by NSFC grant nos. 11401188. Wen Wu is the corresponding author.}

\begin{abstract}
Let $E$ be a subset of a doubling metric space $(X,d)$. We prove that for any   $s\in [0, \dim_{A}E]$, where $\dim_{A}$ denotes the Assouad dimension, there exists a subset $F$ of $E$ such that $\dim_{A}F=s$. We also show that the same statement holds for  the lower  dimension $\dim_L$. 
\end{abstract}

\maketitle
\section{Introduction}

Recently, there have been many works devoted to the study of Assouad and the lower dimensions in fractal geometry, see e.g. \cite{Fraser,FHOR,FO,KLV,KR, Luukkainen,Mackay}. Most of those works concentrated on calculations of Assouad (or the lower) dimension of some self-similar or self-affine sets. In the present paper, we propose to consider some basic properties of Assoud and the lower dimensions. More specifically, we would like to investigate, for a given subset $E$ of a metric space $(X,\rho)$, the accessible values of  $\dim_{A}F$ and $\dim_{L}F$ for subsets $F\subset E$. 
The problem we consider is well understood in the classical cases of Hausdorff, packing and box-counting dimensions. 
Let $E\subset \mathbb{R}^d$ with Hausdorff dimension $s$. Then for any $\alpha \in (0,s)$ there is a  subset $F \subset E$ with Hausdorff dimension $\alpha,$ see Besicovitch \cite{Besicovitch} and Davies \cite{Davies}. 
This accessibility property also holds for packing and upper box-counting dimensions, see Joyce and Preiss \cite{Joycepreiss} for the packing dimension and Feng, Wen and Wu \cite{FWW} for the upper box-counting dimension. However, the lower box-counting dimension does not possess the accessibility property, see \cite{FWW}. For more detailed information about these dimensions, see \cite{Falconer, Mattila}.

In this paper, we show that in a doubling metric space,  the Assouad  and the lower dimensions also have the above accessibility property. 
In the rest of this introduction, we recall the definitions of Assouad and the lower dimensions and state our main results.

The Assouad dimension was introduced by Assouad, see \cite{Assouad1,Assouad2}. Let $(X,\rho)$ be a doubling metric space. Recall that a metric space $X$ is a doubling metric space, if there is $N\in \mathbb{N}$ such that for all $r>0$, any ball of radius $r$ can be covered by a collection
of $N$ balls of radius $r/2$. For $r>0$ and $ E\subset X$, let $N_{r}(E)$ denote the least number of open balls of radius less than or equal to $r$ which can cover the set $E$. The (local) \emph{Assouad 
dimension} of $E \subset X$ is defined as 
\begin{align*}
\dim_{A}^* E=\inf \Big\{s \geq 0  : \exists ~C,\rho>0  & \text{ such that } \forall ~ 0<r<R<\rho,   \\
& \sup_{x \in E}N_r \left(E \cap B(x, R) \right) \leq C \left(\frac{R}{r} \right)^s\Big\}.
\end{align*}
It is clear that the (local) Assouad dimension depends only on the local structure of sets. It is not suitable for the measurement of large scale structures.  Because of this, we introduce the (global) \emph{Assouad dimension} which is defined as
\begin{align*}
\dim_A E=\inf \Big\{s \geq 0  : \exists ~ C>0  &\text{ such that }  \forall ~ 0<r<R,   \\
& \sup_{x \in E}N_r \left(E \cap B(x, R) \right) \leq C \left(\frac{R}{r} \right)^s\Big\}.
\end{align*}
A metric space $X$ is doubling if and only if $\dim_A X< \infty$, see \cite[Proposition 1.15]{Heinonen}. 

Note that for any bounded set $E$, $\dim_A^* E=\dim_A E$. In general, we have $\dim^*_A E \leq \dim_A E$ for any set $E$.  The equality does not necessarily hold  for unbounded sets even in some nice doubling metric space, such as  $\R^d$. For example, the integer lattice $\Z^d$ in $\R^d$ has (local) Assouad dimension 0, while  its (global) Assouad dimension is $d$. Actually, for any $0\leq \alpha<\beta\leq d$, we can find a subset $E\subset\mathbb{R}^{d}$ such that $\alpha=\dim^{*}_{A}E$ and $\dim_{A}E=\beta$ (see Section \ref{section remarks-example}).  

The \emph{lower dimension} can be considered as the dual of Assouad dimension. 
It is defined as follows: 
\begin{align*}
\dim_{L}E  = \sup\Big\{ s \geq 0:\exists ~C,\rho>0 &\textrm{ such that }  \forall ~0<r<R <\rho,\\ 
& \inf_{x\in E}N_{r}(E\cap B(x,R))\geq C\left(R/r\right)^{s}
\Big\}.
\end{align*}
The lower dimension was introduced by Larman, see \cite{Larman}.

 The following Theorem is our main result. 

\begin{theorem}\label{Assouad}
Let $(X,\rho)$ be a doubling metric space and $E\subset X$. For any $\alpha \in [0, \dim_A^* E]$ there exists a subset $F \subset E$ with $\dim_A^* F =\alpha.$ The same result also holds for $\dim_A$ and $\dim_L$.
\end{theorem}

We claim that it is sufficient to prove Theorem \ref{Assouad}  for the case when $X\subset\R^d$ for some $d$ endowed with Euclidean metric. Let $(X, \rho^\varepsilon), ~0 <\varepsilon < 1$ be the snowflake metric space (see \cite[p.3]{Heinonen}). The Assouad embedding theorem (see \cite[Theorem 3.15]{Heinonen}) says that there is a bi-Lipschitz embedding $f$ which maps $(X, \rho^\varepsilon)$ into some Euclidean space $(\R^{d},|\cdot|)$. In fact the map $f$ is also an embedding from $(X, \rho)$ into $(\R^{d},|\cdot|)$. The claim follows by the following easy fact that 
\[
\dim_A (f(K),|\cdot|)=\frac{\dim_A (K,\rho)}{\varepsilon},~ K\subset X
\]
and this equality also holds for $\dim_A^{*}$ and $\dim_L$, here $\dim_A (K,\rho)$ denotes the Assouad dimension of $K$ with respect to the metric $\rho$.


The paper is organized as follows. In Section \ref{section Preliminary}, we recall some basic properties of Assouad and the lower dimensions; some equivalent definitions of these dimensions are given. The claims of Theorem \ref{Assouad} for Assouad dimension are proved in Section \ref{section, continuity property of Assouad dimension}. The lower dimension case is proved in Section  \ref{section, continuity property of lower dimension}. In Section \ref{section remarks-example}, we give some further remarks.

\section{Preliminary} \label{section Preliminary}

In this section, we give some equivalent definitions of Assouad and the lower dimensions which are more convenient to use for our later constructions.

We first introduce an equivalent definition of  Assouad  dimension, called the \emph{star dimension} and denoted by $\dim^{*}$, which was   introduced by Furstenberg in \cite{Furstenberg}. 

 Let $\mathcal{Q}_1$ be the collection of all cubes of $\R^{d}$ with side length $\leq 1$ and $\Q$ be the collection of all cubes of $\R^{d}$.
For any cube $Q\in \mathcal{Q}$, dividing $Q$ into $p^{d}$ $(2\leq p\in \N)$ equal sub-cubes, let $N_{p}(E,Q)$ be the number of those sub-cubes intersecting $E$. 
Defining 
\begin{equation}
H_{p}^{*}(E,\mathcal{Q}_1):=\max_{Q\in\mathcal{Q}_1}N_{p}(E,Q),  \label{e:hstar}
\end{equation} the local star dimension of $E$ is given by 
\begin{equation}
\dim^{*}E:=\lim_{p\rightarrow\infty}\frac{\log H_{p}^{*}(E,\mathcal{Q}_1)}{\log p}=\inf_{p>1} \frac{\log H_{p}^{*}(E,\mathcal{Q}_1)}{\log p} \label{e:sd}.
\end{equation}
Similarly, the (global) star dimension of $E$ is defined as 
\begin{equation}
\dim E:=\lim_{p\rightarrow\infty}\frac{\log H_{p}^{*}(E,\mathcal{Q})}{\log p} =\inf_{p>1} \frac{\log H_{p}^{*}(E,\mathcal{Q})}{\log p}
\end{equation}
where
\[
H_{p}^{*}(E,\mathcal{Q}):=\max_{Q\in\mathcal{Q}}N_{p}(E,Q). 
\]
The existence of the above limits is a simple consequence of the fact that the functions $p\mapsto H_{p}^*(E,Q)$ and $p\mapsto H_{p}^*(E,Q_1)$ are sub-multiplicative (see \cite[page 13]{Furstenberg}).

\begin{lemma}\label{lem:ballandcubes}
For any $E\subset \mathbb{R}^{d}$,  $\dim_A^* E= \dim^* E$ and $\dim_{A}E=\dim E$.
\end{lemma}
\begin{proof}
We only prove the first equality; the second one holds in a similar way. Assume $\dim^{*} E=s$. For any $\varepsilon>0$, there exists $M\in \N$ such that for any $p \geq M$,
\begin{equation}\label{eq:maximalintersection}
p^{s-\varepsilon}\leq H_p^{*}(E, \Q_1) \leq p^{s+\varepsilon}.
\end{equation}
Thus for any $n\in \N$, there is a cube $Q_n\in \Q_1$ with side length $\ell_n$ such that
\begin{equation}\label{eq:v}
N_{M^n}(E,Q_n)\geq M^{n(s-\varepsilon)}.
\end{equation}
Let $x_n \in Q_n\cap E, R_n = \sqrt{d}\ell_n$ and $r_n =M^{-n}\ell_n$. Then by an elementary geometric argument we obtain that 
\begin{equation}\label{a1}
N_{r_n}\left(E\cap B(x_n, R_n)\right)\geq C_1 N_{M^{n}}(E, Q_n)
\end{equation}
for some constant $C_1$ depending only on $d$.
Thus, by \eqref{eq:v} and \eqref{a1} we have
\[
N_{r_n}\left(E\cap B(x_n, R_n)\right) \geq C_{1} M^{n(s-\varepsilon)} \geq C_{2} (R_n/r_n)^{s-\varepsilon},
\]
where $C_{2}$ is a constant depending only on $d$.
Hence $\dim^{*}_A E \geq s-\varepsilon$. Since we can choose $\varepsilon$ arbitrarily small, $\dim^{*}_A E \geq s$.

We show the opposite direction in the following. For any ball $B(x,R)$ with $x\in E$ and $0<r<R\leq 1.$ There exists a unique cube $Q$, which contains $B(x,R)$, of side length $2R$. Let $p= \lfloor \frac{R}{r}\rfloor+1$.
First we assume that $p\geq M$. In this case, applying the estimate \eqref{eq:maximalintersection}, $N_p(E, Q)\leq p^{s+\varepsilon}$, and by a volume argument we obtain that  
\[
N_{2r\sqrt{d}}(E\cap B(x,R)) \leq p^{s+\varepsilon}\leq  C_{3}\left(\frac{R}{r}\right)^{s+\varepsilon}.
\]
Hence 
\[
N_{r}(E\cap B(x,R)) \leq  C_{4}\left(\frac{R}{r}\right)^{s+\varepsilon}.
\] 

For the case $p<M$, there exits a $C_{5}$ only depending on $d$ such that
\[
N_{r}(E\cap B(x,R))\leq N_{r}(B(x,R))\leq C_{5} \leq C_{5}\left(\frac{R}{r}\right)^{s+\varepsilon}.
\]
By the arbitrary choice of $\varepsilon$, we conclude that $\dim^{*}_A E\leq s$. Thus we complete the proof.
\end{proof}

In the definition of star dimension, it is convenient to consider $b$-adic cubes where $b\geq 2$ is an integer. 
\begin{lemma}
Let $\mathcal{D}_{b}$ be the set of all $b$-adic  cubes in $[0,1]^{d}$. Then
\[\dim^{*}E=\lim_{n\rightarrow\infty}\frac{\log H^{*}_{b^{n}}(E,\mathcal{D}_{b})}{n\log b}.\]
\end{lemma} 
\begin{proof}
Since the limit in (\ref{e:sd}) exists, the limit can be reached using the subsequence $p=b^{n}$, i.e., \[\dim^{*}E=\lim_{n\rightarrow\infty}\frac{\log H_{b^{n}}^{*}(E,\mathcal{Q})}{\log b^{n}}.\]  
Note that $\mathcal{D}_{b}\subset  \mathcal{Q}$. By (\ref{e:hstar}), we have $H_{b^{n}}^{*}(E,\mathcal{D}_{b})\leq H_{b^{n}}^{*}(E,\mathcal{Q})$ for any $n\geq 1$. On the other hand, for any $Q\in \mathcal{Q}$ of side length $\ell$, it can be covered by at most $2^{d}$ cubes of side length $b^{m}$ in $\mathcal{D}_{b}$  where $b^{-m-1}\leq \ell <b^{-m}$. Suppose $Q\subset \bigcup_{i=1}^{k}D_{i}$, where $D_{i}\in\mathcal{D}_{b}$ and $k\leq 2^{d}$. Thus \[N_{b^{n}}(E,Q)\leq \sum_{i=1}^{k}N_{b^{n}}(E,D_{i})\leq 2^{d}\max_{D\in \mathcal{D}_{b}}N_{b^{n}}(E,D).\]
Since $Q$ can be chosen arbitrarily, we have $H_{b^{n}}^{*}(E,\mathcal{D}_{b})\geq 2^{-d}H_{b^{n}}^{*}(E,\mathcal{Q})$. Therefore
\[\dim^{*}E=\lim_{n\rightarrow\infty}\frac{\log H_{b^{n}}^{*}(E,\mathcal{Q})}{\log b^{n}}=\lim_{n\rightarrow\infty}\frac{\log H_{b^{n}}^{*}(E,\mathcal{D}_{b})}{\log b^{n}}.\]
Thus we complete the proof.
\end{proof}

Now we give an equivalent definition of the lower dimension. For any $A\subset \mathbb{R}^{d}$ and $ a\in A$, let $N^{*}_{r}(E\cap B(a,R))$ be the largest number of disjoint balls of radius  $r$ and centres in $E$ contained in $B(a,R)$. Such balls are called \emph{packing balls}.
\begin{lemma}\label{lem:lowereq}
For any $E\subset \mathbb{R}^{d}$, 
\begin{equation}
\begin{aligned}
\dim_{L}E  = \sup\Big\{ s:  \exists\ C,\rho>0 \textrm{ s.t. } & \forall ~0<r<R<\rho,\\ 
& \inf_{x\in E}N^{*}_{r}(E\cap B(x,R))\geq C\left(R/r\right)^{s}
\Big\}.
\end{aligned}
\end{equation}
\end{lemma}
\begin{proof}
The result follows directly from the fact that for any $A\subset \mathbb{R}^{d}, a\in A$ and $0<r<R$, we have 
\[
N_{2r}(A\cap B(a, R/2))\leq N_{r}^{*}(A \cap B(a,R))\leq N_{r/3}(A\cap B(a,R)).
\] 
The first inequality holds because if we enlarge the radius of $N_{r}^{*}(A\cap B(a,R))$ packing balls to $2r$, then the enlarged $N_{r}^{*}(A\cap B(a,R))$ balls must cover $A\cap B(a, R/2)$. Otherwise there exists a ball of radius $r$ and centre in $A\cap B(a, R/2)$ which is disjoint with the previous $N_{r}^{*}(A\cap B(a,R))$ packing balls. This contradicts the definition of $N_{r}^{*}(A\cap B(a,R))$. For the second inequality, since the centres of $N_{r}^{*}(A\cap B(a,R))$ packing balls are inside $B(a,R)$, each $\frac{r}{3}$-cover of $A\cap B(a,R)$ must cover such centres. Hence each packing ball contains at least one $\frac{r}{3}$-covering ball. 
\end{proof}

\section{ Proof of Theorem \ref{Assouad} for Assouad dimension}
\label{section, continuity property of Assouad dimension}

Recall that $\mathcal{D}_{b}$ is the set of all $b$-adic  cubes in $[0,1]^{d}$.
For $n\in \N$, let $\mathcal{D}_{b}(n)$ be the collection of all $b$-adic cubes of side length $b^{-n}$. For any $K\subset \mathcal{D}_{b}(n)$, $\widetilde{K}$ will stand for the union of cubes $\bigcup_{D\in K}D$, which is a subset of $\mathbb{R}^{d}$. 
\begin{lemma} \label{lem:rc}
Let $\varepsilon>0$,  $M,n\in \N$ and $K \subset \D_{M}(n)$ such that  ${\rm Card} (K) \geq M^{ns}$ and $N_M(\widetilde{K},Q) \leq M^{s+\varepsilon}$ for any $Q \in \mathcal{D}_{M}(i), 1\leq i \leq n-1$. Let $N\in \N$ so that $N \leq M^{s+\varepsilon}$. Then there exists $F \subset K$ such that 
\begin{equation}
N_M(\widetilde{F},Q) \leq N \label{e:l:1}
\end{equation}
for any $Q \in \D_{M}(i-1)$ and $1\leq i\leq n$,
and furthermore we have
\begin{equation}
{\rm Card} (F) \geq N^nM^{-n\varepsilon}.  \label{e:l:2}
\end{equation}
\end{lemma}
\begin{proof}
We will construct a sequence of sets $F_i \ (i=0,1,\cdots, n)$ which are composed of cubes in $\mathcal{D}_{M}(i)$, i.e., $F_{i}\subset \mathcal{D}_{M}(i)$. Moreover, $\widetilde{F}_{i+1}\subset \widetilde{F}_{i}$ for $i=0,1,\cdots,n-1$ and the set $F_n$ satisfies \eqref{e:l:1} and \eqref{e:l:2}.

Let $F_0=\{[0,1]^d\}$. Suppose that $F_{i}$ has been  constructed. To construct $F_{i+1}$, for each $Q\in F_{i}$, we select some sub-cubes $\mathcal{B}(Q)\subset \mathcal{A}(Q)$ with cardinality $\min\{N,{\rm Card}(\mathcal{A}(Q))\}$  where 
$$\mathcal{A}(Q):=\{C\in \D_M(i+1): C\subset Q \text{ and } C\cap\widetilde{K}\neq \emptyset \}.$$
Let $F_{i+1}=\bigcup_{Q\in F_i}\mathcal{B}(Q).$ We claim that there is one such $F_n$ which satisfies \eqref{e:l:1} and \eqref{e:l:2}. Note that \eqref{e:l:1} is always true by construction. To that end, we use a probabilistic method and assume that each $\mathcal{B}(Q)\subset \mathcal{A}(Q)$ was chosen uniformly at random (so that each $\min\{N,{\rm Card}(\mathcal{A}(Q))\}$ element subset of $\mathcal{A}(Q)$ has the same probability of being chosen).  Then for each $Q\in K$ there is probability at least $(N/M^{s+\varepsilon})^{n}$ for $Q\in F_n$.
By the linearity of expectation,  
\begin{equation}
\mathbb{E}( {\rm Card}  (F_n))= \sum_{Q \in K} \mathbb{P}(Q \in F_n) \geq N^nM^{-n\varepsilon}.
\end{equation}
So there exists an $F_n$ with ${\rm Card}  (F_n) \geq N^nM^{-n\varepsilon}$. 
\end{proof}


\begin{lemma}\label{lem:dense}
Let $E\subset [0,1]^{d}$ with $\dim^{*}E=s\in [0,d]$. For any $\alpha\in (0,s)$ and any $\varepsilon>0$ there exists an $F\subset E$ such that $\dim^{*}F\in [\alpha-\varepsilon, \alpha+\varepsilon]$.
\end{lemma}
\begin{proof}
Since $\dim^{*}E=s$, there exists an $M_0 \in \N$, such that for any $M \geq M_0$, 
\begin{equation}\label{eq:lem:dense:1}
H_M^*(E, \D_M) \leq M^{s+\varepsilon/2}.
\end{equation}
Let $N := \lfloor M^\alpha \rfloor$ be the integer part of $M^\alpha$. In the following we suppose $M$ is  large enough such that 
$$N\ge M^{\alpha-\varepsilon/2}.$$ 
By the definition of $\dim_A E$, we have $H_p^*(E, \mathcal{Q}_1) \geq p^{s}$ holds for any $p\geq 2$.

Let $Q_{0}=[0,1]^d$. By finite stability of Assouad dimension, there exists a cube $Q_{1} \in \mathcal{D}_{M}(1)$ such that $\dim^{*}(Q_{1}\cap E)=s$. Similarly we obtain $Q_2 \in \D_{M}(2),~ Q_2 \subset Q_{1}$ with $\dim^{*}(Q_{2}\cap E)=s$. In the end we obtain a sequence  $(Q_{n})_{n\geq 1}\subset \mathcal{D}_{M}$ such that $Q_{n+1}\in \mathcal{D}_{M}(n+1),~ Q_{n+1} \subset Q_{n}$ and $\dim_{A}(Q_{n}\cap E)=s$ for all $n\geq 1$. 

For any $n$, since $\dim^{*}(E\cap Q_{n})=s$, there exist  $i_{n}\geq n$ and $I_n \in \mathcal{D}(i_{n}),~ I_n \subset Q_n$ such that
\begin{equation}
M^{ns}\leq N_{M^{n}}(E, I_{n})\leq M^{n(s+\varepsilon/2)}. \label{e:thm1:1}
\end{equation}
Let 
\[
K_{n}=\left\{Q\in \mathcal{D}_M(n+i_{n}) ~|~ Q\cap E\cap I_n\neq \emptyset \right\}.
\]
By \eqref{eq:lem:dense:1} and (\ref{e:thm1:1}), $K_{n}$ satisfies the condition of Lemma  \ref{lem:rc}. Applying Lemma \ref{lem:rc} to $K_{n}$, we obtain an $F_n\subset K_{n}$ satisfying (\ref{e:l:1}) and (\ref{e:l:2}). From each cube $Q$ in $F_{n}$, we arbitrarily choose one point in $Q\cap E$  and denote the union of such points by $E_{n}$. We also conclude that $E_{n}$ satisfies
\begin{equation}
H^{*}_{M^{n}}(E_{n},\mathcal{D})\geq N^{n}M^{-n\varepsilon/2}\geq M^{n(\alpha-\varepsilon)}. \label{e:t:en}
\end{equation}

Set $n_1= 1$ and $n_{k+1}=n_k+i_{n_k}$ for all $k \geq 1$.
Let $E^{\prime} = \bigcup^\infty_{k=1} E_{n_k}$.
We show that $\dim^{*} E^{\prime} \geq \alpha-\varepsilon.$ For any $k\geq 1$, by (\ref{e:t:en}), we have \[H_{M^{n_k}}^*(E^{\prime},\D) \geq H_{M^{n_k}}^*(E_{n_k},\D)\geq M^{n_k(\alpha-\varepsilon)}.\] Thus $\dim^{*} E^{\prime} \geq \alpha-\varepsilon$.

On the other hand, for any $Q\in \D_M$ there exists $n_k\leq \ell < n_{k+1}$ such that $Q\in \D_M(l)$ and 
\[
E'=(\cup^k_{j=1} E_{n_j}) \cup (\cup^\infty_{j=k+1} E_{n_j})\subset (\cup^k_{j=1} E_{n_j}) \cup Q_{n_{k+1}}. 
\] 
Observe that
\[
N_{M}^*(\cup_{j=1}^{k}E_{n_j},Q)\leq N
\] 
and 
\[
N_{M}^*(\cup_{j=k+1}^{\infty}E_{n_j},Q)\leq N_{M}^*(Q_{n_{k+1}},Q)\leq 3^{d}.
\] 
Thus 
\begin{equation}\label{induct}
N_{M}^*(E^{\prime},Q) \leq N+3^{d}\leq M^{\alpha+\varepsilon},
\end{equation}
the last estimates holds when $M$ is large.
and hence $\dim^{*} E^{\prime} \leq \alpha+\varepsilon.$
\end{proof}

\begin{lemma}\label{lem:unboundedforlocal}
Let $E\subset \mathbb{R}^{d}$ be an unbounded set with $\dim^{*} E=s$ and for any $R>0$, $\dim^{*} (E \cap B(0,R))<s$. Then for any $\alpha \in (0,s)$ and $\varepsilon>0$ there exists an $F\subset E$ such that $\dim^{*} F\in (\alpha-\varepsilon, \alpha+\varepsilon)$.
\end{lemma}
\begin{proof}
By definition of $\dim^{*}$, we have 
$$\dim^{*}(E)=\sup_{Q\in \mathcal{Q}_1}\dim^{*}(E\cap Q).$$
Thus for any $0<\alpha<s$, there exists $Q\in \mathcal{Q}_1$ such that $\dim^{*}(E\cap Q)\geq\alpha$.
By Lemma \ref{lem:dense}, there is $F\subset E\cap Q$ with $\dim^{*} F\in (\alpha-\varepsilon, \alpha+\varepsilon)$.
\end{proof}

\begin{lemma}\label{lem:unboundedforgloble}
Let $E\subset \mathbb{R}^{d}$ be unbounded with $\dim E=s$ and for any $R>0$, $\dim (E \cap B(0,R))<s$. Then for any $\alpha\in (0,s)$ and  $\varepsilon>0$ there exists an $F\subset E$ such that $\dim F\in [\alpha-\varepsilon, \alpha+\varepsilon]$.
\end{lemma}
\begin{proof}
Since $\dim E=s$, there exists $M \in \N$ such that for any $p \geq M$ 
\[
p^{s-\varepsilon} \leq H_p^* (E,\Q) \leq p^{s+\varepsilon}.
\]
Thus there exists $Q_1 \in \Q$ such that $N_M(E,Q_1)\geq M^{s-\varepsilon}$. By Lemma \ref{lem:rc} there exists a finite set $F_1 \subset Q_1$ such that 
\[
 M^{\alpha-\varepsilon}\leq N_M(F_1,Q_1) \leq M^{\alpha+\varepsilon}.
\]
From each cube $Q$ in $F_{1}$, we arbitrarily choose one point in $Q\cap E$  and denote the union of such points by $F_{1}'$.
Given  $\{F_i'\}^k_{i=1}$ and $\{Q_i\}^{k}_{i=1}$, we construct $F_{k+1}'$ as follows. Let  $R_{k}=\sup_{x\in Q_k}{\rm dist}(0,x)$ and $\ell_{k}$ be a positive number  satisfying   
\begin{equation}\label{eq: condition11}
\sum_{i=1}^{k} \diam(Q_{i})^{\alpha+\varepsilon} \leq \ell_{k}^{\alpha+\varepsilon}.
\end{equation}
In addition, we require that $\ell_1 < \ell_2 <\cdots < \ell_k.$

Since $\dim (E \backslash B(0, R_{k}+\ell_{k}))=s,$
there exists $Q_{k+1}$ and $n_{k+1}$ such that 
\[
N_{M^{n_{k+1}}}(E,Q_{k+1})\geq M^{n_{k+1}(s-\varepsilon)}
\]
and $Q_{k+1} \cap B(0, R_k+\ell_{k})=\emptyset$. By Lemma \ref{lem:rc} there exists $F_{k+1} \subset Q_{k+1}$ such that for any $Q \subset Q_{k+1}$ and $Q_{k}\in \mathcal{D}(k+2)$
\begin{equation}\label{eq:up}
N_M(F_{k+1},Q) \leq M^{\alpha+\varepsilon}. 
\end{equation}
Furthermore, we have
\begin{equation}
N_{M^{n_{k+1}}}(F_{k+1},Q_{k+1})\geq M^{n_{k+1}(\alpha-\varepsilon)}.
\end{equation}
From each cube $Q$ in $F_{n+1}$, we arbitrarily choose one point in $Q\cap E$  and denote the union of such points by $F_{n+1}'$.
Define $F=\bigcup_{k\in \N} F_k'.$ It is clear that $\dim F\geq \alpha-\varepsilon.$ In the following we intend to show $\dim  F \leq \alpha+\varepsilon.$

For any $x \in F$  and $r<R$, there exist $k, n$ such that $\ell_k \leq R<\ell_{k+1}$ and $x\in F_n'$. We have two cases:

{\bf Case 1.} $n\geq k+2$.  In this case the ball $B(x,R)$ does not intersect $F_{j}'$ for any $j\neq n$, so by \eqref{eq:up} we have 
\[
N_r(F \cap B(x,R))= N_r(F_{n}' \cap B(x,R))  \leq C (R/r)^{\alpha+\varepsilon}.
\]

{\bf Case 2.}  $n< k+2$. In this case, the ball $B(x,R)$ does not intersect $F_j'$ for $j\geq k+2$. So we have
\begin{align*}
N_r(F \cap B(x,R))  &\leq \sum_{i=1}^{k}
N_r( F_i')+N_r( F_{k+1}'\cap B(x,R))\\
& \leq  \sum_{i=1}^{k}
N_r( Q_i)+N_r( F_{k+1}'\cap B(x,R))\\
&\leq C\sum^{k}_{i=1} \left(\dfrac{\diam(Q_i)}{r}\right)^{\alpha+\varepsilon} +C\left(\dfrac{R}{r}\right)^{\alpha+\varepsilon}\\
&\leq 2C \left(\frac{R}{r}\right)^{\alpha+\varepsilon},
\end{align*}
where  the last inequality holds due to  the condition \eqref{eq: condition11} and the fact $\ell_k \leq R$.
Thus we complete the proof.
\end{proof}

It is clear that the star dimension is not countable stable. For example let $A_n=\{1/k\}^n_{k=1},$ then 
\[
0=\sup_{n\in \N} \dim^{*} A_n <\dim^{*} \left(\bigcup_{n\in \N} A_n\right)=1.
\]
However, we have the following easy fact. For convenience we put it as a lemma.
\begin{lemma}\label{lem:twosidesapproch}
Let $A_n \subset A_{n+1}, B_{n+1}\subset B_n$ for all $n\in \N$ and $\bigcup_{n\geq 1} A_{n}\subset\bigcap_{n\geq 1} B_{n}$. Then 
\begin{equation}
\sup_{n\in \N} \dim^{*} A_n\leq \dim^{*}\left( \bigcup^\infty_{n=1} A_n\right)\leq \dim^{*} \left(\bigcap^{\infty}_{n=0} B_n\right) \leq \inf_{n\in \N} \dim^{*} B_n.
\end{equation}
The above formula also holds for $\dim$.
\end{lemma}

Applying Lemma \ref{lem:unboundedforlocal} and Lemma \ref{lem:twosidesapproch}, we intend to prove Theorem \ref{Assouad} for $\dim^{*}$ and $\dim$.

\begin{proof}[Proof of Theorem \ref{Assouad} for $\dim^{*}$]
Let $\alpha \in (0,s)$. Choose two sequences $\{a_n\}_{n\in \N}$ with $a_n \nearrow \alpha$ and $\{b_n\}_{n\in \N}$ with $b_n \searrow s$. Let 
\[
 (0,\alpha)= \bigcup_{n\in \N} I_n  \text{ and } (\alpha, s)=\bigcup_{n\in \N} J_n,
\]
where $I_n=(a_n, a_{n+1}]$, and $ J_n=[b_{n+1},b_n)$. For $n=1$, by Lemma \ref{lem:unboundedforlocal} there is $B_1 \subset E$ with $\dim B_1 \in J_1$, and $A_1 \subset B_1$ with $\dim^* A_1 \in I_1$. Given $A_n, B_n$ with $A_n \subset B_n$ and $\dim^* A_n \in I_n, \dim^* B_n \in J_n$. We intend to construct $A_{n+1}, B_{n+1}$. For $B_n$ and $J_{n+1}$, by Lemma  \ref{lem:unboundedforlocal}, there is $B_n' \subset B_n $ with $\dim^* B_n' \in J_{n+1}$. For $B_n'$ and $I_{n+1}$, by Lemma \ref{lem:unboundedforlocal}, there is $A_n'\subset B_n'$ with $\dim^* A_n' \in I_{n+1}$. Let 
\begin{equation}
A_{n+1}=A_n \cup A_n', B_{n+1} = A_n \cup B_n'.
\end{equation}
In the end  we have two sequences $\{A_n\}, \{B_n\}$ with  $$A_n \subset A_{n+1}\subset B_{n+1}\subset B_n$$ and $\dim^* A_n \in I_n, \dim^* B_n \in J_n$ for all $n\in \N$. By lemma \ref{lem:twosidesapproch}, we have $\dim^* \left(\cup_n A_n\right) = \dim^* \left(\cap_n B_n\right)=\alpha$.
\end{proof}

\begin{proof}[Proof of Theorem \ref{Assouad} for $\dim$]
Applying Lemma \ref{lem:unboundedforgloble}, Lemma \ref{lem:twosidesapproch} and the same argument as in the previous proof, we complete the proof.
\end{proof}

\section{Proof of Theorem \ref{Assouad} for the lower dimension}\label{section, continuity property of lower dimension}

\begin{proof}[Proof of Theorem \ref{Assouad} for $\dim_L$]
Suppose $E\subset \mathbb{R}^{d}$ with $\dim_{L}E=s$. For any $\alpha\in (0,s)$, we will construct a subset $F$ which satisfies $\dim_{L}F=\alpha$ as follows.\\
{\bf Step 1.} Let $\varepsilon\in(0,s-\alpha)$. Then by Lemma \ref{lem:lowereq}, there exist $C,\rho>0$ such that for any $0<r<R<\rho$,
\[N^{*}_{r}(E\cap B(x,R))\geq C\left(\frac{R}{r}\right)^{s-\varepsilon}.\]

Let $M$ be a large enough integer satisfying $$CM^{\frac{s-\varepsilon}{\alpha}}\geq M+3^{d}.$$ Fix $0<R_0<\rho$. Let $R_{1}=\lambda R_{0}$ where $\lambda^{\alpha}M=1$. Then for all  $x\in E$, 
\[N^{*}_{R_{1}}(E\cap B(x,R_{0}))\geq C\lambda^{-s+\varepsilon}=CM^{\frac{s-\varepsilon}{\alpha}}\geq M+3^{d}.\] 
Fix any $x\in E$. Since there are at most $3^{d}$ disjoint balls of radius $R_{1}$ that touch $B(x,R_{1})$, we can pick $M$ disjoint balls including $B(x,R_{1})$ from the $N^{*}_{R_{1}}(E\cap B(x,R_{0}))$ packing balls. Denote them by $\{B(x_{i},R_{1})\}_{i=1}^{M}$ where $x_{1}=x$. Let $\mathcal{A}:=\{1,\dots, M\}$ and $$E_{1}=\bigcup_{i\in\mathcal{A}}E\cap B(x_{i},R_{1}).$$ 
{\bf Step 2.} Let $R_{2}=\lambda R_{1}$. Repeating Step 1 for each $i\in\mathcal{A}$, we choose $M$ balls including $B(x_{i}, R_{2})$ from $N^{*}_{R_{2}}(E\cap B(x_{i},R_{1}))$ packing balls. Denote them by $B(x_{\mathbf{i}}, R_{2})$ where $\mathbf{i}=i_{1}i_{2}\in \mathcal{A}^{2}$ and $x_{i1}=x_{i}$ for all $i\in\mathcal{A}$. Let \[E_{2}=\bigcup_{\mathbf{i}\in \mathcal{A}^{2}}E\cap B(x_{\mathbf{i}},R_{2}).\]
{\bf Step 3.} Repeating previous process $n$ times, we define \[E_{n}=\bigcup_{\mathbf{i}\in\mathcal{A}^{n}}E\cap B(x_{\mathbf{i}},R_{n})\]
and \[F=\bigcap_{n\geq 1}E_{n}.\] 
This process ensures that
\begin{enumerate}
\item $R_{k+1}=\lambda R_{k}$ for all $k\geq 1$;
\item $\{B(x_{\mathbf{i}},R_{k})\}_{\mathbf{i}\in\mathcal{A}^{k}}$ are disjoint balls for all $k\geq 1$;
\item $B(x_{\mathbf{i}j},R_{k+1})\subset B(x_{\mathbf{i}},R_{k})$ for all $\mathbf{i}\in\mathcal{A}^{k}, j\in\mathcal{A}$ and $k\geq 1$;
\item $x_{\mathbf{j}1}=x_{\mathbf{j}}$ for all $\mathbf{j}\in\mathcal{A}^{k}$. Thus $x\in F\neq \emptyset$.
\end{enumerate} 
By definition of lower dimension, it is not hard to see that the lower dimension of a bounded set is always less than its lower box dimension.
Therefore, \[\dim_{L}F\leq \underline{\dim}_{B}F\leq \lim_{n\rightarrow\infty}\frac{\log M^{n}}{-\log(R_{0}\lambda^{n})}=\alpha.\]

In the following, we will show that $\dim_{L}F\geq \alpha$. Let $C_{M}=\frac{1}{M+1}$. For any $0<r<R\leq R_{0}$, we have \[R_{n+1}<R\leq R_{n} \textrm{ and } R_{n+k+1}<r\leq R_{n+k}\]
for some $n\geq 1$ and $k\geq 0$.  When $k=0$, $R/r<R_{n}/R_{n+1}=1/\lambda$. Thus by the definition of $F$, we know that for any $x\in F$,
\begin{align*}
N_{r}^{*}(F\cap B(x,R)) & \geq 1 \\
& \geq C_{M}M = C_{M}\frac{1}{\lambda^{\alpha}}\\
& \geq C_{M}\left(\frac{R}{r}\right)^{\alpha}.
\end{align*}
For the same reason, when $k\geq 1$, we have 
\begin{align*}
N_{r}^{*}(F\cap B(x,R)) & \geq N_{r}^{*}(F\cap B(x,R_{n+1}))\\
& \geq N_{R_{n+k}}^{*}(F\cap B(x,R_{n+1}))\\
& \geq M^{k-1}=\lambda^{-\alpha(k-1)}=\left(\frac{R_{n+1}}{R_{n+k}}\right)^{\alpha}\\
& \geq C_{M}\left(\frac{R_{n+1}}{R_{n+k}}\right)^{\alpha}.
\end{align*}
Now we conclude that $\dim_{L}F\geq \alpha$ which completes the proof.
\end{proof}

\section{Further remarks}\label{section remarks-example}

The equality $\dim_A E =\dim^{*}_A E$ holds for any bounded set $E$ in Euclidean space. However this is not always true in general metric space. For example for any infinite set $X$, given the discrete metric on $X$, that is any two different points has distance one, thus 
\[
0=\dim^{*}_A E<\dim_A E =\infty.
\]

For unbounded sets (even in Euclidean spaces), these two dimensions can be different.
\begin{proposition}
For any $0\leq \alpha < \beta \leq d$ there exists a subset $E \subset \mathbb{R}^d,$ such that 
\[
\alpha=\dim^{*}_{A} E < \dim_A E =\beta.
\]
\end{proposition}
\begin{proof}
By Theorem \ref{Assouad}, we can find two sets $K_{1}\subset\mathbb{Z}^{d}$ and $K_{2}\subset [0,1]^{d}$ with 
\[\dim_{A}K_{1}=\beta \text{ and }\dim^{*}_{A}K_{2}=\alpha.\]
Note that $\dim^{*}_{A}K_{1}=0$ and $\dim_{A}K_{2}=\alpha$. By the finite stability of $\dim_{A}$ and $\dim^{*}_{A}$, we have
\[\alpha=\dim_{A}^{*}(K_{1}\cup K_{2})<\dim_{A}(K_{1}\cup K_{2})=\beta.\]
This completes the proof.
\end{proof}

\medskip
\noindent{\bf Concluding remark.} After this work was essentially completed, we are informed that Theorem \ref{Assouad} for the (local) Assouad dimension $\dim_{A}^{*}$ is also obtained independently by W. Wang and S. Y. Wen in \cite{WW} at the same time. We are grateful to them for providing us their manuscript.

\medskip
\noindent{\bf Acknowledgement.} The authors are  grateful to DeJun Feng,  
Esa J\"arvenp\"a\"a, ChiuHong Lo and Ville Suomala  for many fruitful and interesting discussions related to this work.

\end{document}